\documentclass[a4paper,12pt,final]{amsart}
\usepackage{times,a4wide,mathrsfs,amssymb,dsfont}

\newcommand{\C}{\mathbb{C}}
\newcommand{\Z}{\mathbb{Z}}

\newcommand{\QQ}{\mathbb{Q}}

\newcommand{\PP}{\mathbb{P}}
\newcommand{\LL}{\mathbb{L}}
\newcommand{\A}{\mathbb{A}}

\newcommand{\XX}{\mathcal X}

\newcommand{\MM}{\mathcal M}

\newcommand{\gr}{\hbox{Gr}}

\newcommand{\ima}{\hbox{Im}}
\newcommand{\rom}{\romannumeral}

\newcommand{\ide}{\hbox{id}}

\newtheorem{theorem}{Theorem}[section]
\newtheorem{claim}[theorem]{Claim}
\newtheorem{lemma}[theorem]{Lemma}
\newtheorem{sublemma}[theorem]{Sublemma}
\newtheorem{corollary}[theorem]{Corollary}

\newtheorem{nonumbering}{Theorem}

\newtheorem{convention}{Conventions}

\theoremstyle{definition}
\newtheorem{remark}[theorem]{Remark}
\newtheorem{definition}[theorem]{Definition}

\newtheorem{nonumberingt}{Acknowledgements}

\begin{document}
\author[Robert Laterveer]
{Robert Laterveer}

\address{Institut de Recherche Math\'ematique Avanc\'ee,
CNRS -- Universit\'e 
de Strasbourg,\
7 Rue Ren\'e Des\-car\-tes, 67084 Strasbourg CEDEX,
FRANCE.}
\email{robert.laterveer@math.unistra.fr}

\title[On the motive of Kapustka--Rampazzo's CY threefolds]{On the motive of Kapustka--Rampazzo's Calabi-Yau threefolds}

\begin{abstract} Kapustka and Rampazzo have exhibited pairs of Calabi-Yau threefolds $X$ and $Y$ that are L--equivalent and derived equivalent, without being birational.
We complete the picture by showing that $X$ and $Y$ have isomorphic Chow motives.
\end{abstract}



\maketitle

\section{Introduction}

Let $\hbox{Var}(\C)$ denote the category of algebraic varieties over the field $\C$.
The Grothendieck ring $K_0(\hbox{Var}(\C))$ encodes fundamental properties of the birational geometry of varieties. The intricacy of the ring $K_0(\hbox{Var}(\C))$
is highlighted by the result of Borisov \cite{Bor}, showing that the class of the affine line $\LL$ is a zero--divisor in $K_0(\hbox{Var}(\C))$. Following on Borisov's pioneering result, a great many people have been hunting for Calabi--Yau varieties 
 $X, Y$ that are {\em not\/} birational (and so $[X]\not=[Y]$ in the Grothendieck ring), but
  \[          ([X] -[Y]) \LL^r=0\ \ \ \hbox{in}\ K_0(\hbox{Var}(\C)) \ ,\]
  i.e., $X$ and $Y$ are ``L--equivalent'' in the sense of \cite{KS}.
  In many cases, the captured varieties $X$ and $Y$ are also derived equivalent \cite{IMOU}, \cite{IMOU2}, \cite{Mar}, \cite{Kuz}, \cite{OR}, \cite{BCP}, \cite{KS}, \cite{HL}, \cite{Man}, \cite{KR}, \cite{KKM}.
  
   According to a conjecture made by Orlov \cite[Conjecture 1]{Or}, derived equivalent smooth projective varieties should have isomorphic Chow motives. This conjecture 
   is true for $K3$ surfaces \cite{Huy}, but is still
   open for Calabi--Yau varieties of dimension $\ge 3$. In \cite{L}, I verified Orlov's conjecture for the Calabi--Yau threefolds of Ito--Miura--Okawa--Ueda \cite{IMOU}.
  The aim of the present note is to check that Orlov's conjecture is also true for the threefolds constructed recently by Kapustka--Rampazzo:
  
  \begin{nonumbering}[=theorem \ref{main}] Let $X, Y$ be two derived equivalent Calabi--Yau threefolds as in \cite{KR}. Then there is an isomorphism of Chow motives
   \[ h(X)\cong h(Y)\ \ \ \hbox{in}\ \MM_{\rm rat}\ .\]
   \end{nonumbering}
   
To prove theorem \ref{main}, we exploit the ``homological projective duality--style'' diagram given in \cite{KR} relating $X$ and $Y$. One key ingredient in the proof that might be of independent interest is a result (theorem \ref{higher}) concerning higher Chow groups of certain fibrations; this is a variant of a result of Vial's \cite{V4}.

\vskip0.6cm

\begin{convention} In this note, the word {\sl variety\/} will refer to a reduced irreducible scheme of finite type over the field of complex numbers $\C$. For any variety $X$, we will denote by $A_j(X)$ the Chow group of dimension $j$ cycles on $X$ 
with $\QQ$--coefficients.
For $X$ smooth of dimension $n$, the notations $A_j(X)$ and $A^{n-j}(X)$ will be used interchangeably. 

The notation 
$A^j_{hom}(X)$ will be used to indicate the subgroups of 
homologically trivial cycles.
For a morphism between smooth varieties $f\colon X\to Y$, we will write $\Gamma_f\in A^\ast(X\times Y)$ for the graph of $f$, and ${}^t \Gamma_f\in A^\ast(Y\times X)$ for the transpose correspondence.

The contravariant category of Chow motives (i.e., pure motives with respect to rational equivalence as in \cite{Sc}, \cite{MNP}) will be denoted $\MM_{\rm rat}$. 
\end{convention}

 \section{The Calabi--Yau threefolds}
 
 \begin{theorem}[Kapustka--Rampazzo \cite{KR}]\label{kr} Let $X,Y$ be a general pair of Calabi--Yau threefolds in the family $\bar{\XX}_{25}$ that are dual to one another (in the sense of \cite[Section 2]{KR}). Then $X$ and $Y$ are not birational, and so
   \[ [X]\not= [Y]\ \ \ \hbox{in}\ K_0(\hbox{Var}(\C)) \ .\]
  
  However, one has
   \[  ([X] -[Y]) \LL^2=0\ \ \ \hbox{in}\ K_0(\hbox{Var}(\C)) \ .\]
   
   Moreover, $X$ and $Y$ are derived equivalent, i.e. there is an isomorphism of bounded derived categories
     \[ D^b(X)\cong D^b(Y)\ .\]
     
 In particular, there is an isomorphism of polarized Hodge structures
     \[ H^3(X,\Z)\ \cong\ H^3(Y,\Z)\ .\]
    \end{theorem}

     \begin{proof} Everything but the last phrase is in \cite{KR}. The isomorphism of Hodge structures is a corollary of the derived equivalence, in view of
     \cite[Proposition 2.1 and Remark 2.3]{OR}.
      \end{proof}
    
%
  
 \begin{remark} As explained in \cite{KR}, the threefolds $X,Y$ in the family $\bar{\XX}_{25}$ are a limit case of the Calabi--Yau threefolds in the family $\XX_{25}$ studied in \cite{BCP}, \cite{OR}.
 A pair of dual varieties $X,Y$ in the family $\XX_{25}$ are also derived equivalent and L-equivalent (the exponent of $\LL$ is, however, higher than in theorem \ref{kr}).  
 \end{remark}

 \section{Higher Chow groups and fibrations}
 
 \begin{definition}[Bloch \cite{B2}, \cite{B3}] Let $\Delta^j\cong\A^j(\C)$ denote the standard $j$--simplex. For any quasi--projective variety $M$ and any $i\in\Z$, let $z_i^{simp}(M,\ast)$ denote the simplicial complex where $z_i(X,j)$ is the group of $(i+j)$--dimensional algebraic cycles in $M\times\Delta^j$ that meet the faces properly. Let $z_i^{}(M,\ast)$ denote the single complex associated to $z_i^{simp}(M,\ast)$. The higher Chow groups of $M$ are defined as
    \[ A_i(M,j):= H^j( z_i^{}(M,\ast)\otimes\QQ)\ .\]
 \end{definition}
 
 \begin{remark} Clearly one has $A_i(M,0)\cong A_i(M)$. Higher Chow groups are related to higher algebraic $K$--theory: there are isomorphisms
   \begin{equation}\label{K} \gr_\gamma^{n-i} K_j(M)_\QQ\cong A_i(M,j)  \ \ \ \hbox{for\ all\ }i,j \end{equation}
   where $K_j(M)$ is Quillen's higher $K$--theory group associated to the category of coherent sheaves on $M$, and $\gr^\ast_\gamma$ is a graded for the $\gamma$--filtration \cite{B2}. Higher Chow groups are also related to Voevodsky's motivic cohomology (defined as hypercohomology of a certain complex of Zariski sheaves) \cite{Fr}, \cite{MVW}.
  \end{remark}

 For later use, we establish the following result, which is a variant of a result of Vial's \cite{V4}:
 
 \begin{theorem}\label{higher} Let $\pi\colon M\to B$ be a flat projective morphism between smooth quasi--projective varieties of relative dimension $m$. Assume that for every $b\in B$, the fibre $M_b:=\pi^{-1}(b)$ has 
    \[ A_i(M_b)=\QQ\ \ \ \forall i\ .\]
    
   \noindent
   (\rom1) 
  The maps
  \[  \Phi_\ast:= \sum_{k=0}^{m} h^{m -k}\circ \pi^\ast\colon\ \ \ \bigoplus_{k=0}^{m} A_{\ell -k}(B,j)\ \to\ A_{\ell}(M,j) \]
  and
  \[  \Psi_\ast:= \sum_{k=0}^{m} \pi_\ast\circ  h^{k}\colon\ \ \   A_{\ell}(M,j)\ \to\   \bigoplus_{k=0}^{m} A_{\ell -k}(B,j) \]    
  are both isomorphisms, for any $\ell$ and $j$.
  (Here $h^k$ denotes the operation of intersecting with the $k$--th power of a hyperplane section $h\in A^1(M)$.)
  
  \noindent
  (\rom2) Set $V_k:=(\Psi_\ast)^{-1} A_{\ell-k}(B,j)\ \subset A_\ell(M,j)$. Then
    \[  (\Phi_\ast \Psi_\ast)\vert_{V_{m}}   = \lambda\, \ide\ ,\]
    for some non--zero $\lambda\in\QQ$.
  
  \end{theorem}
 
 \begin{proof} 
 
 \noindent(\rom1)
 For $j=0$ (i.e., for usual Chow groups), this is exactly \cite[Theorem 3.2]{V4}. For arbitrary $j$ (i.e., for higher Chow groups), a straightforward although laborious proof would consist in convincing the reader that everything Vial does in the proof of \cite[Theorem 3.2]{V4} also applies to higher Chow groups. Indeed, all formal properties of Chow groups exploited in loc. cit. also hold for higher Chow groups.
 
 Under the simplifying assumption that all fibres $M_b$ are isomorphic to $\PP^m$ (which will be the case when we apply theorem \ref{higher} in this note), a quick proof could be as follows. Let $H\subset M$ be a general hyperplane section, and let $U=\subset B$ be the open over which the fibres of the restricted morphism $\pi\vert_H\colon H\to B$ are isomorphic to $\PP^{m-1}$. Let $M_U:=\pi^{-1}(U)$, and let us consider the restricted morphism 
   \[\pi\vert_U\colon M_U\to U\ .\] 
   Using the localization sequence for higher Chow groups and noetherian induction, we are reduced to proving (\rom1) for $\pi\vert_U$.
   Let us consider the open $M^\prime_U:=M_U\setminus (H\cap M_U)$. The fibres of the morphism $\pi^\prime\colon M^\prime_U\to U$ are isomorphic to $\A^m$. There is a commutative diagram with exact rows
   \[ \begin{array}[c]{ccccccc}
      \to&  A_i(M^\prime_U,j+1) &\to& A_i(H,j) &\to& A_i(M_U,j) &\!\!\!\to\\
      &&&&&&\\
     & \uparrow{\scriptstyle (\pi^\prime)^\ast } &&\ \ \ \ \ \ \ \ \ \ \ \ \ \  \ \ \ \   \uparrow{\scriptstyle \sum_{k=0}^{m-1} h^{m-1-k}\circ(\pi\vert_H)^\ast }&&\ \ \ \ \ \ \ \ \ \ \ \ \ \  \  \uparrow{\scriptstyle \sum_{k=0}^m h^{m-k}\circ (\pi\vert_U)^\ast }&\\
     &&&&&&\\
     \to& A_i(U,j+1)&\to& {\displaystyle \bigoplus_{k=0}^{m-1}} A_{k}(U,j) &\to& {\displaystyle \bigoplus_{k=0}^{m}} A_{k}(U,j)    &\!\!\!\to \\
         \end{array}\]
   Doing an induction on the fibre dimension $m$, it will suffice to prove that $(\pi^\prime)^\ast$ is an isomorphism for all $i,j$. But this follows from the corresponding result for $K$--theory \cite[Proposition 4.1]{Q}, in view of the isomorphism (\ref{K}) and the fact that the pullback $(\pi^\prime)^\ast\colon K_j(U)\to K_j(M^\prime_U)$ respects the $\gamma$--filtration. This proves that $\Phi_\ast$ is an isomorphism. The argument for $\Psi_\ast$ is similar.

 \noindent
 (\rom2) 
 The direct summand $V_{m}$ can be identified as
   \[ V_{m}=\bigcap_{k=0}^{m-1} \ker \bigl( \pi_\ast\circ h^{k}\bigr)\ \ \ \ \subset\ A_\ell(M,j)\ .\]
   Using this description, it is readily checked that 
   \[ V_{m}= \pi^\ast A_{\ell-m}(X,j)\ .\]
   This implies (\rom2).
 
 \end{proof} 
  
 \begin{remark} 
 In case $B$ and $M$ are smooth projective, theorem \ref{higher} can be upgraded to a relation of Chow motives \cite[Theorem 4.2]{V4}. 
 In the more general case where $B$ and $M$ are smooth but only quasi--projective, perhaps one can relate $B$ and $M$ in the category $DM^{eff}_{gm}$ of Voevodsky motives ? If so, the relation of higher Chow groups obtained in theorem \ref{higher} would be an immediate consequence, since higher Chow groups (with $\QQ$--coefficients) can be expressed as Hom--groups in $DM^{eff}_{gm}$ \cite{Fr}, \cite{MVW}. 
 \end{remark}

\section{Main result}     

\begin{theorem}\label{main} Let $X,Y$ be a pair of Calabi--Yau threefolds in the family $\bar{\XX}_{25}$ that are dual to one another, in the sense of \cite[Section 2]{KR}. 
 Then there is an isomorphism 
   \[ h(X)\cong h(Y)\ \ \ \hbox{in}\ \MM_{\rm rat}\ .\]
  \end{theorem} 
     
  \begin{proof} 
  First, to simplify matters, let us slightly cut down the motives of $X$ and $Y$. It is known \cite{KR} that $X$ and $Y$ have Picard number $1$. A routine argument 
  gives a decomposition of the Chow motives
  \[ \begin{split}   h(X)&= \mathds{1} \oplus \mathds{1}(1)\oplus h^3(X) \oplus \mathds{1}(2) \oplus \mathds{1}(3)\ ,\\
                   h(Y)&= \mathds{1} \oplus \mathds{1}(1)\oplus h^3(Y) \oplus \mathds{1}(2) \oplus \mathds{1}(3)\ \ \ \ \ \ \hbox{in}\ \MM_{\rm rat}\ ,\\
                  \end{split} \] 
  where $\mathds{1}$ is the motive of the point $\hbox{Spec}(k)$. (The gist of this ``routine argument'' is as follows: let $H\in A^1(X)$ be a hyperplane section. Then 
    \[ \pi^{2i}_X:= c_i H^{3-i}\times H\ \ \ \in\ A^3(X\times X)\ , \ \ \ 0\le i\le 3\ ,\] 
    defines an orthogonal set of projectors lifting the K\"unneth components, for appropriate $c_i\in\QQ$. One can then define $\pi^3_X=\Delta_X-\sum_i \pi^{2i}_X\in A^3(X\times X)$, and $h^j(X)=(X,\pi^i_X,0)\in \MM_{\rm rat}$, and ditto for $Y$.)
  
  To prove the theorem, it will thus suffice to prove there is an isomorphism of motives
    \begin{equation}\label{iso3}   h^3(X)\cong h^3(Y)\ \ \ \hbox{in}\ \MM_{\rm rat}\ .\end{equation}
    We observe that the above decomposition (plus the fact that $H^\ast(h^3(X))=H^3(X)$ is odd--dimensional) implies equality
     \[ A^\ast(h^3(X)) = A^\ast_{hom}(X)\ ,\]
     and similarly for $Y$.
    
%
%
    
To construct the isomorphism (\ref{iso3}), we need look no further than the construction of the threefolds $X, Y$. As explained in \cite[Section 2]{KR}, the Calabi--Yau threefolds $X, Y$ are related via a diagram
  \begin{equation}\label{diag} \begin{array}[c]{ccccccccc}   && D & \xrightarrow{i} & M & \xleftarrow{j} & E && \\
                &&&&&&&&\\
                   &{}^{\scriptstyle f_X} \swarrow \ \ && {\scriptstyle f} \swarrow \ \ \ &\downarrow & \ \ \ \searrow {\scriptstyle g} & & \ \ \searrow {}^{\scriptstyle g_Y} & \\
                   &&&&&&&&\\
                   X & \hookrightarrow & G(2,5) & \xleftarrow{\pi_1} & F & \xrightarrow{\pi_2}& G(3,5) & \hookleftarrow & Y\\
                   \end{array}\end{equation}
                   
         Here $G(j,5)$ denotes the Grassmannian of $j$--dimensional subspaces in a $5$--dimensional vector space. The variety $F$ is the flag variety parametrizing 
         pairs $(V,W)\in G(2,5)\times G(3,5)$ such that $V\subset W$. The variety $M\subset F$ is a hyperplane section. The Calabi--Yau varieties $X,Y$ are closed subvarieties of $G(2,5)$ resp. $G(3,5)$, and the closed subvarieties $D,E$ are defined as $f^{-1}(X)$ resp. $g^{-1}(Y)$.
         The morphisms $f,g$ are $\PP^1$--fibrations over the opens $G(2,5)\setminus X$ resp. $G(3,5)\setminus Y$, but the restrictions $f_X, g_Y$ are $\PP^2$--fibrations.
         
   The flag variety $F$ has trivial Chow groups (i.e. $A^\ast_{hom}(F)=0$), and so $F$ has a Chow--K\"unneth decomposition (this is a general fact for any smooth projective variety with trivial Chow groups: since all cohomology is algebraic, a K\"unneth decomposition exists; since $F\times F$ again has trivial Chow groups, the K\"unneth decomposition is a Chow--K\"unneth decomposition). By a standard trick (cf. for instance \cite[Lemma 5.2]{IMS}), this induces a Chow--K\"unneth decomposition $\{\pi^j_M\}$ for the hyperplane section $M\subset F$, with the property that 
   \[ (M,\pi^j_M)\cong \oplus \mathds{1}(\ast)\ \ \ \hbox{in}\ \MM_{\rm rat}\ \ \ \hbox{for\ all\ }j\not=7=\dim M\ .\]
   In particular, we have that
   \[  A^i_{hom}(M)=  A^i(h^7(M)):= (\pi^7_M)_\ast A^i(M)\ \ \ \hbox{for\ all\ }i\ .\]
   
   We now make the following claim: 
   
   \begin{claim}\label{cla}
   There are isomorphisms
   \[   \begin{split}   \Gamma_1\colon\ \ h^3(X)\ &\xrightarrow{\cong}\ h^7(M)(2)\ ,\\
                                                                              \Gamma_2\colon\ \ h^3(Y)\ &\xrightarrow{\cong}\ h^7(M)(2)\ \ \ \ \ \hbox{in}\ \MM_{\rm rat}\ .\ \\      
                                                                              \end{split}\]
    \end{claim}                                                                          
                                                                              
                                                                         This claim obviously suffices to prove (\ref{iso3}). To prove the claim, let us treat the isomorphism $\Gamma_1$ in detail (the same argument applies to $\Gamma_2$, upon replacing $X$ and $G(2,5)$ by $Y$ resp. $G(3,5)$). To prove the claim for $\Gamma_1$, it will suffice to find correspondences $\Gamma_1\in A^{5}(X\times M), \Psi_1\in A^{5}(M\times X)$ with the property that
                            \begin{equation}\label{isoact} \begin{split}   (\Xi_1)_\ast(\Gamma_1)_\ast=\ide\colon\ \ \ A^i_{hom}(X)\ &\to\ A^i_{hom}(X)\ ,\\        
                                                                                               (\Gamma_1)_\ast (\Xi_1)_\ast=\ide\colon\ \ \ A^i_{hom}(M)\ &\to\ A^i_{hom}(M)\ .\\
                                                                                               \end{split}
                                                                                               \end{equation}
                                   (Indeed, let us assume one has correspondences $\Gamma_1,\Xi_1$ satisfying (\ref{isoact}). By what we have said above, this means that
                          \begin{equation}\label{iso2} \begin{split}   & (\pi^3_X\circ\Xi_1\circ \pi^7_M\circ\Gamma_1\circ\pi^3_X)_\ast  =(\pi^3_X)_\ast\colon\ \ \   
                             A^i_{}(X)\ \to\ A^i_{}(X)\ ,\\ 
                                                                     &    (\pi^7_M\circ\Gamma_1\circ \pi^3_X\circ\Xi_1\circ\pi^7_M)_\ast  =(\pi^7_M)_\ast\colon\ \ \  A^i(M)\ \to\ A^i_{}(M)\ .\\
                                                                                               \end{split}
                                                                                               \end{equation}

                                   There exists a field $k\subset\C$, finitely generated over $\QQ$, such that $X,M,\pi^j_X,\pi_M^j\Gamma_1,\Xi_1$ are defined over $k$. Because $\C$ is a universal domain, for any finitely generated field extension $K\supset k$, there is an inclusion $K\subset \C$. Thus, the natural maps $A^i(X_K)\to A^i(X_\C)$ and   $A^i(M_K)\to A^i(M_\C)$ are injections \cite[Appendix to Lecture 1]{B}. This implies that the relations (\ref{iso2}) also hold over $K$. Manin's identity principle then gives that 
                  \[ \Gamma_1\colon\ \  \ h^3(X_k)\ \to\ h^7(M_k)(2)\ \ \ \hbox{in}\ \MM_{\rm rat} \]
                  is an isomorphism, and so $\Gamma_1$ induces an isomorphism of motives over $\C$ as claimed.)
       
Before proving the claim, let us introduce some lemmas.

\begin{lemma}\label{l1} Set--up as above. The composition
  \[ A^i_{hom}(X)\ \xrightarrow{(f_X)^\ast}\ A^i_{hom}(D)\ \xrightarrow{i_\ast}\ A^{i+2}_{hom}(M) \]
  is surjective, for any $i$.
\end{lemma}

\begin{proof} Let us write $U:=M\setminus D$, and $G:=G(2,5)$. By assumption, $U$ is a $\PP^1$--fibration over $V:=G\setminus X$. 

 For any $i$, there is a commutative diagram with exact rows
  \[ \begin{array}[c]{cccccccc}
      \ \  \to & A_i(V,1) &\to & A_i(X) &\to& A_i(G) &\to&    A_i(V) \\
        &\downarrow  &&\downarrow&&\downarrow&&\downarrow {}\\
       0 \to & W^{-2i} H_{2i-1}(V,\QQ)\cap F_{-i} &\to& H_{2i}(X,\QQ)\cap F_{-i} &\to& H_{2i}(G,\QQ)&\to& W^{-2i} H_{2i}(V,\QQ)\\
        \end{array}\]
        where vertical arrows are (higher) cycle class maps into Borel--Moore homology, and $W^\ast, F_\ast$ denote the weight filtration resp. the Hodge filtration on Borel--Moore homology \cite{PS}.
    (The upper row is exact thanks to localization for higher Chow groups \cite{B2}, \cite{B3}, \cite{Lev}. The lower row is exact because the category of polarizable pure Hodge structures is semisimple \cite{PS}. For the cycle class map from higher Chow groups into Borel--Moore homology, cf. \cite[Section 4]{Tot}.)   
The Grassmannian $G$ has trivial Chow groups. 
Using the fact that the Hodge conjecture is true for the threefold $X$, this implies that the cycle class map induces isomorphisms
  \[ A^i(V)\ \xrightarrow{\cong}\ W ^{-2i} H^{2i}(V,\QQ)\ ,\]
  and the higher cycle class map induces a surjection
  \[   A_i(V,1)\ \twoheadrightarrow\   W^{1-2i}H_{2i-1}(V,\QQ)\cap F_{1-2i}\ .\]
  (These two facts together can be paraphrased by saying that $V$ satisfies a variant\footnote{It is a variant, because in \cite{Tot} only the weight filtration and not the Hodge filtration is taken into account. This works fine for the linear varieties considered in \cite{Tot}, but not for the varieties $U,V$ under consideration here.} of the ``strong property'' of Totaro's \cite[Section 4]{Tot}.)
  Using theorem \ref{higher}, plus the corresponding property of cohomology, this implies that $U$ has the same property (i.e., $U$ satisfies the strong property).
  
  For any $i$, there is a commutative diagram with exact rows
  \[ \begin{array}[c]{cccccccc}
       \to & A_i(U,1) &\to & A_i(D) &\to& A_i(M) &\to&    A_i(U) \\
        &\downarrow  &&\downarrow&&\downarrow&&\downarrow {}\\
        \to & W^{-2i} H_{2i-1}(U,\QQ)\cap F_{-i} &\to& H_{2i}(D,\QQ)\cap F_{-i} &\to& H_{2i}(M,\QQ)&\to& W^{-2i} H_{2i}(U,\QQ)\\
        \end{array}\]
          By what we have just observed (the strong property for $U$), the left vertical arrow is a surjection and the right vertical arrow is an isomorphism. A quick diagram chase then reveals that pushforward induces a surjection
       \begin{equation}\label{DM} i_\ast\colon\ \ \ A_i^{hom}(D)\ \twoheadrightarrow\ A_i^{hom}(M) \ \ \ \forall i\ .\end{equation}
       
   Next, since $f_X\colon D\to X$ is a $\PP^2$--fibration, theorem \ref{higher} ensures that there are 
   isomorphisms
    \[ \begin{split}  \Phi_\ast:=  \sum_{ k=0}^2  h^{2-k}\circ  (f_X)^\ast\colon\ \  \bigoplus_{k=0}^2 A_{i-k}^{hom}(X)\ &\xrightarrow{\cong}\ A_i^{hom}(D)\ ,\\  
     \Psi_\ast:=  \sum_{ k=0}^2   (f_X)_\ast\circ h^k\colon\ \ A_i^{hom}(D)\ &\xrightarrow{\cong}\ \bigoplus_{k=0}^2 A_{i-k}^{hom}(X)\ .
       \end{split}\]
     We write
     \begin{equation}\label{decom}  \begin{split} A_i^{hom}(D)&= V_0\oplus V_1\oplus V_2\\ &:= (\Psi_\ast)^{-1}A_i^{hom}(X) \oplus   (\Psi_\ast)^{-1}A_{i-1}^{hom}(X)  \oplus (\Psi_\ast)^{-1}A_{i-2}^{hom}(X)\ .\\\end{split}
     \end{equation}
     
  To prove the lemma, it remains to understand the pushforward map (\ref{DM}).   
Precisely, we will show that one summand of the decomposition (\ref{decom}) already surjects onto $A_i^{hom}(M)$: 
  \begin{equation}\label{summands}\    
  \ima \bigl( V_0\oplus V_1\ \xrightarrow{i_\ast}\ A_i^{hom}(M)\bigr) \ \subset\  \ima \bigl( V_2\ \xrightarrow{i_\ast}\ A_i^{hom}(M)\bigr)  
     \ \ \ \hbox{for\ all\ }i\ .\end{equation}
  To see this, we observe that there is a commutative diagram of complexes 
  \[ \begin{array}[c]{ccccc}
           z_i(D,\ast) &\to& z_i(M,\ast) &\to& z_i(U,\ast)\\
           \downarrow&&\downarrow&&\downarrow\\
           z_i(X,\ast)&\to& z_i(G(2,5),\ast)&\to&z_i(V,\ast)\\
           \end{array}\]
         (where the vertical arrows are proper pushforward maps). This gives rise to a commutative diagram with long exact rows
      \begin{equation}\label{diag1}  \begin{array}[c]{ccccccc}
              \to &A_i(U,1) &\xrightarrow{\delta}& A_i(D) &\to& A_i(M)&\to\\
             &\ \ \ \ \ \  \downarrow{\scriptstyle (f_U)_\ast}&&\ \ \ \ \ \downarrow{\scriptstyle (f_X)_\ast}&&\ \ \ \ \downarrow{\scriptstyle f_\ast}&\\
              \to &A_i(V,1) &\xrightarrow{\delta^\prime}& A_i(X) &\to& A_i(G(2,5))&\to\\
              \end{array}\end{equation}
  
Let us now assume $b\in A_i^{hom}(D)$ lies in the summand $V_0$ of the decomposition (\ref{decom}). Then $(f_X)_\ast(b)$ is in $A_i^{hom}(X)$.
Since $A_i^{hom}(G(2,5))=0$, this means that $(f_X)_\ast(b)$ is in the image of the map $\delta^\prime$, say $(f_X)_\ast(b)=\delta^\prime(c^\prime)$. In view of theorem \ref{higher}, the element $c^\prime\in A_i(V,1)$ comes from an element $c\in A_i(U,1)$ lying in the direct summand (isomorphic to) $A_i(V,1)$. Using sublemma \ref{compat} below, this means that there is equality
  \[ \delta(c) = b - b_2\ \ \ \hbox{in}\ A_i(D)\ ,\]
  for some $b_2\in A_i(D)$ lying in the summand (isomorphic to) $A_{i-2}(X)$. It follows that
   \[ i_\ast(b)= i_\ast(b_2)\ \ \in\ \ima\Bigl(  A_{i-2}(X)\to A_i(D)\to A_i(M)\Bigr)\ .\]
   As $i_\ast(b_2)=i_\ast(b)$ is homologically trivial, the surjection (\ref{DM}) above shows that we may suppose $b_2$ is homologically trivial, and so we have found $b_2$ lying in the summand denoted $V_2$ (isomorphic to $A_{i-2}^{hom}(X)$). This shows that
   \[ i_\ast(b)\ \ \in\ \ima  \Bigl(  A_{i-2}^{hom}(X)\to A_i(D)\to A_i(M)\Bigr)=: \ima\bigl( V_2\ \to\ A_i(M)\bigr)  \ .\]

  Let us next assume that $b\in A_i^{hom}(D)$ lies in the summand $V_1$ of the decomposition (\ref{decom}). The commutative diagram of complexes up to quasi--isomorphism
    \[ \begin{array}[c]{ccccc}
           z_i(D,\ast) &\to& z_i(M,\ast) &\to& z_i(U,\ast)\\
           \downarrow {\scriptstyle h} &&\downarrow{\scriptstyle h}&&\downarrow{\scriptstyle h}\\
           z_{i-1}(D,\ast) &\to& z_{i-1}(M,\ast) &\to& z_{i-1}(U,\ast)\\
           \downarrow&&\downarrow&&\downarrow\\           
           z_{i-1}(X,\ast)&\to& z_{i-1}(G(2,5),\ast)&\to&z_{i-1}(V,\ast)\\
           \end{array}\]  
  gives rise to a commutative diagram with exact rows
     \begin{equation}\label{diag2} \begin{array}[c]{ccccccc}
              \to &A_i(U,1) &\xrightarrow{\delta}& A_i(D) &\to& A_i(M)&\to\\
              &&&&&&\\
             & \ \ \ \ \ \ \downarrow { \scriptstyle (f\vert_U)_\ast \circ h}   &&\ \ \ \ \ \ \downarrow {\scriptstyle (f_X)_\ast\circ  h} &&\ \ \ \ \ \ 
             \downarrow {\scriptstyle  f_\ast\circ  h} &\\
              &&&&&&\\
              \to &A_{i-1}(V,1) &\xrightarrow{\delta^\prime}& A_{i-1}(X) &\to& A_{i-1}(G(2,5))&\to\\
              \end{array}\end{equation}
  
  Reasoning just as above, we can find $c\in A_i(U,1)$ lying in the summand (isomorphic to) $A_{i-1}(V,1)$ such that
   \[  \delta(c) = b - b_2\ \ \ \hbox{in}\ A_i(D)\ ,\]
  where $b_2\in A_i(D)$ is in the summand (isomorphic to) $A_{i-2}(X)$. It follows once more that
   \[ i_\ast(b)= i_\ast(b_2)\ \ \in\ \ima\Bigl(  A_{i-2}(X)\to A_i(D)\to A_i(M)\Bigr)\ ,\]  
  and (using the surjectivity (\ref{DM})) that
  \[   i_\ast(b)\ \ \in\ \ima  \Bigl(  A_{i-2}^{hom}(X)\to A_i(D)\to A_i(M)\Bigr)=: \ima\bigl( V_2\ \to\ A_i(M)\bigr) \ .\]
  We have now proven the inclusion (\ref{summands}).

  Combining (\ref{DM}), (\ref{summands}) and theorem \ref{higher}(\rom2), we see that there is a surjection
  \[ A_{i-2}^{hom}(X)  \ \twoheadrightarrow\ A^{hom}_i(M)\ ,\]
  which is given by $i_\ast (f_X)^\ast$. This proves the lemma.

\end{proof}

In the proof of lemma \ref{l1} we have used the following sublemma:

\begin{sublemma}\label{compat} Given $i\in\Z$, let 
  \[  \Psi_\ast\colon\ \ A_i(D)=\bigoplus_{k=0}^2 A_{i-k}(X)\ ,\ \ \ \Psi_\ast\colon\ \ A_i(U,1)=\bigoplus_{k=0}^1 A_{i-k}(V,1)\ \]  
 be the decompositions of theorem \ref{higher}. Let $\delta\colon A_i(U,1)\to A_i(D)$ be the boundary map of the 
 localization exact sequence for the inclusion $D\subset M$. Then
  \[ \begin{split} \delta\bigl(  A_i(V,1)\bigr)\ &\subset\ A_i(X)\oplus A_{i-2}(X) \ ,\\
                              \delta\bigl(  A_{i-1}(V,1)\bigr)\ &\subset\ A_{i-1}(X)\oplus A_{i-2}(X) \ .\\     \end{split} \]
  \end{sublemma}

\begin{proof} For the first inclusion, we consider the commutative diagram (\ref{diag2}). In view of theorem \ref{higher}, the direct summand of $A_i(U,1)$ isomorphic to $A_i(V,1)$ is exactly the kernel of the map $(f\vert_U)_\ast \circ h$. As such, the image under $\delta$ is contained in
  \[ \ker \Bigl( A_i(D)\ \xrightarrow{ (f_X)_\ast\circ h}\ A_{i-1}(X)\Bigr)\ .\]
 Again applying theorem \ref{higher}, this kernel coincides with the two summands isomorphic to $A_i(X)$ resp. to $A_{i-2}(X)$, as claimed.
 
 The second inclusion is proven similarly, reasoning in the diagram (\ref{diag1}).
\end{proof}

\begin{lemma}\label{l2} Set--up as above. There is equality
  \[ D =\lambda h^2  +   h\cdot f^\ast(d_1) + f^\ast(d_2)  \ \ \ \hbox{in}\ A^2(M)\ ,\]
  for some non--zero $\lambda\in\QQ$ and some $d_i\in A^i(G(2,5))$, $i=1,2$.
\end{lemma}

\begin{proof} Let us consider the restriction $h^2\vert_U$ of $h^2\in A^2(M)$ to the open $U:=M\setminus D$. Let $f_U\colon U\to V$ be the restriction of the morphism $f$, where $V:=G(2,5)\setminus X$. As we have seen, $f_U$ is a $\PP^1$--fibration. It thus follows from theorem \ref{higher} that
  \[  h^2\vert_U =  h\cdot (f_U)^\ast(c_1) + (f_U)^\ast(c_2)\ \ \ \hbox{in}\ A^2(U)\ ,\]
  for some $c_i\in A^i(V)$, $i=1,2$.
  Let $\bar{c}_i\in A^i(G(2,5))$ be elements such that $\bar{c}_i\vert_U=c_i$ for $i=1,2$. The localization exact sequence (plus the fact that $D$ is irreducible of codimension $2$ in $M$) then implies that
  \begin{equation}\label{mu}  h^2 = h\cdot f^\ast(\bar{c}_1) + f^\ast(\bar{c}_2) +  \mu D  \ \ \ \hbox{in}\ A^2(M) \ ,\end{equation}
  for some $\mu\in\QQ$. 
  
  Let us assume, for a moment, that $\mu=0$. Then relation (\ref{mu}) would imply in particular that
  \[  h^2\vert_D =  \Bigl(h\cdot f^\ast(\bar{c}_1) + f^\ast(\bar{c}_2) \Bigr)\vert_D\ \ \ \hbox{in}\ A^2(D)\ .\]
  But this is absurd, for the right hand side maps to $0$ under pushforward $(f_X)_\ast$ whereas the left hand side maps to a non--zero multiple of $[X]\in A_3(X)$ under pushforward 
  $(f_X)_\ast$. It follows that $\mu\not=0$.
  
  Relation (\ref{mu}) proves the lemma; it suffices to define $\lambda:=1/\mu$ and $d_i:= \lambda \, \bar{c}_i\in A^i(G(2,5))$, $i=1,2$.
  \end{proof}

Armed with these lemmas, we are now ready to prove the claim \ref{cla} (and hence close the proof of the theorem). 
Let $d\in\Z$ be the non--zero integer such that $(f_X)_\ast(h^2)=d[X]$ in $A_3(X)$.
We define correspondences $\Gamma_1, \Xi_1$ as follows:
   \[ \begin{split}   \Gamma_1&:=  \Gamma_i\circ {}^t \Gamma_{f_X}        \ \ \ \in\ A^5(X\times M)\ ,\\
                            \Xi_1&:={1\over d\lambda} \, {}^t \Gamma_1 = {1\over d\lambda}\,   \Gamma_{f_X}\circ {}^t \Gamma_i   \ \ \ \in\ A^{5}(M\times X)\ \\
                           \end{split} \]
               (where $\lambda$ is the non--zero constant of lemma \ref{l2}). 
               
               Let us show these correspondences $\Gamma_1, \Xi_1$ verify
               the relations (\ref{isoact}).
                              By construction, the composition $\Xi_1\circ \Gamma_1$ acts on Chow groups in the following way:
          \[ (\Xi_1\circ\Gamma_1)_\ast\colon\ \ \     A_i^{}(X)\xrightarrow{(f_X)^\ast} A_{i+2}^{}(D) \xrightarrow{i_\ast} A_{i+2}^{}(M) 
           \xrightarrow{{1\over d\lambda}i^\ast} A_i^{}(D)
          \xrightarrow{(f_X)_\ast} A_i^{}(X)\ .\]     
         
     Thanks to lemma \ref{l2}, the map 
     \[  {1\over d\lambda}i^\ast i_\ast\colon \ \ A_{i+2}(D)\ \to\  A_i(D)       \]
     is the same as intersecting with 
       \[   {1\over d}\, \Bigl( h^2 +  {1\over\lambda}(f_X)^\ast(d_1\vert_X)\cdot h + {1\over\lambda}(f_X)^\ast(d_2\vert_X)\Bigr)\ \ \ \in \ A^2(D)\ .\] 
       In particular, if $b\in A_i(X)$ then
     \[     {1\over d\lambda}i^\ast i_\ast (f_X)^\ast(b)  ={1\over d}\, \Bigl(h^2\circ(f_X)^\ast(b)     + {1\over\lambda}h\circ (f_X)^\ast(b\cdot d_1\vert_X) +{1\over\lambda}(f_X)^\ast(b\cdot d_2\vert_X)\Bigr) \ \ \ \hbox{in}\ A_{i}(D)\ .\]
     But then, it follows that
     \[     \begin{split}   (f_X)_\ast {1\over d\lambda}i^\ast i_\ast (f_X)^\ast(b)  &= {1\over d}\, (f_X)_\ast \Bigl(h^2\circ(f_X)^\ast(b)     
         + {1\over\lambda}h\circ (f_X)^\ast(b\cdot d_1\vert_X) +{1\over\lambda}(f_X)^\ast(b\cdot d_2\vert_X) \Bigr)\\
                      &= {1\over d} (f_X)_\ast \bigl(h^2\circ(f_X)^\ast(b)\bigr)\\ 
                        &=    {1\over d}  (f_X)_\ast(h^2)\cdot  b      =b
     \ \ \ \hbox{in}\ A_{i}(X)\ .\\
     \end{split}\]
     That is, $\Xi_1\circ \Gamma_1$ acts as the identity on $A_i(X)$, which proves the first half of the claimed result (\ref{isoact}).
     
     It remains to prove the second half of (\ref{isoact}). The composition $\Gamma_1\circ\Xi_1$ acts on Chow groups in the following way:
     \[  (\Gamma_1\circ\Xi_1)_\ast\colon\ \ \  A_i^{hom}(M)\xrightarrow{{1\over d\lambda}i^\ast}A_{i-2}^{hom}(D)\xrightarrow{(f_X)_\ast} A_{i-2}^{hom}(X)
         \xrightarrow{(f_X)^\ast}
           A_i^{hom}(D)\xrightarrow{i_\ast} A_i^{hom}(M)\ .\]
           
     Let $a\in A_i^{hom}(M)$. In view of lemma \ref{l1}, we may suppose $a=i_\ast(f_X)^\ast (b)$, for some $b\in A_{i-2}^{hom}(X)$. But we have just checked that
     $(\Xi_1\circ\Gamma_1)_\ast(b)=b$ for any $b\in A_{i-2}(X)$, which means that
     \[  (f_X)_\ast {1\over d\lambda} i^\ast(a)=  (f_X)_\ast {1\over d\lambda} i^\ast i_\ast(f_X)^\ast(b) = b\ \ \ \hbox{in}\    A_{i-2}^{hom}(X)\ .\]
     Applying $i_\ast(f_X)^\ast$ on both sides, we conclude that
     \[ (\Gamma_1\circ\Xi_1)_\ast(a)= i_\ast(f_X)^\ast   (f_X)_\ast {1\over d\lambda} i^\ast i_\ast(f_X)^\ast(b) =  i_\ast(f_X)^\ast   b=a\ \ \ \hbox{in}\    A_{i}^{hom}(M)\ ,\]   
     i.e., $\Gamma_1\circ\Xi_1$ acts as the identity on $A^{hom}_i(M)$
     as claimed.
     
     We have now established the equalities (\ref{isoact}), and so we have proven the first half of claim \ref{cla}. The second half of claim \ref{cla} (i.e., the existence of the isomorphism $\Gamma_2$) is proven by the same argument, the only difference being that $X$ and $G(2,5)$ should be replaced by $Y$ resp. $G(3,5)$.
        \end{proof}   
        
\begin{remark} It would be interesting to refine theorem \ref{main} to an isomorphism with $\Z$--coefficients. Is it true that there are isomorphisms
  \[ \ \ \ A^i_{}(X)_{\Z}\ \xrightarrow{\cong}\ A^i_{}(Y)_{\Z}\ \ \ \forall i \]
  of Chow groups with $\Z$--coefficients ?
  
  The problem, in proving this, is that the fibration result (theorem \ref{higher}) is a priori only valid for (higher) Chow groups with rational coefficients.
\end{remark}

\begin{remark} It would also be interesting to prove theorem \ref{main} for a dual pair $(X,Y)$ of Calabi--Yau threefolds in the family $\XX_{25}$ of \cite{BCP}, \cite{OR}. In the absence of a nice diagram like (\ref{diag}) linking $X$ and $Y$, this seems considerably more difficult than theorem \ref{main}.
\end{remark}

\section{A corollary}

\begin{corollary}\label{cor} Let $X,Y$ be the Calabi--Yau threefolds constructed as in \cite{KR}. Let $M$ be any smooth projective variety. Then there are isomorphisms
  \[ N^j H^i(X\times M,\QQ)\cong N^j H^i(Y\times M,\QQ)\ \ \ \hbox{for\ all\ }i,j\ .\]
  (Here, $N^\ast$ denotes the coniveau filtration \cite{BO}.)
\end{corollary}

\begin{proof} Theorem \ref{main} implies there is an isomorphism of Chow motives $h(X\times M)\cong h(Y\times M)$. As the cohomology and the coniveau filtration only depend on the motive \cite{AK}, \cite{V1}, this proves the corollary.

\end{proof}

\begin{remark} It is worth noting that for any derived equivalent threefolds $X,Y$, there are isomorphisms
  \[ N^j H^i(X,\QQ)\cong N^j H^i(Y,\QQ)\ \ \ \hbox{for\ all\ }i,j\ ;\]
this is proven in \cite{ACV}.
 \end{remark}

\vskip1cm
\begin{nonumberingt}
This note was written during a stay at the Schiltigheim Math Research Institute. Thanks to its director, Mrs. Ishitani, for running the institute with iron hands gloved in velvet.
\end{nonumberingt}

\end{document}